%% file: main.tex
\documentclass[11pt,reqno]{amsart}
\usepackage[a4paper,margin=25mm]{geometry}
\usepackage[T1]{fontenc}
\usepackage{lmodern}
\usepackage{amsmath,amssymb,amsthm,mathtools,mathrsfs,booktabs,microtype}
\usepackage[hidelinks]{hyperref}
\usepackage{enumitem}
\usepackage{tikz}
\setlist[enumerate]{itemsep=3pt,topsep=5pt}
\numberwithin{equation}{section}
\newtheorem{theorem}{Theorem}[section]
\newtheorem{lemma}[theorem]{Lemma}
\newtheorem{proposition}[theorem]{Proposition}
\newtheorem{corollary}[theorem]{Corollary}
\theoremstyle{definition}

\newtheorem{example}[theorem]{Example}
\theoremstyle{remark}

\newcommand{\E}{\mathbb E}

\newcommand{\D}{\operatorname{D}}
\newcommand{\Var}{\operatorname{Var}}
\newcommand{\Deg}{\operatorname{Deg}}
\newcommand{\dist}{\operatorname{dist}}
\newcommand{\diam}{\operatorname{diam}}
\newcommand{\artanh}{\operatorname{artanh}}
\newcommand{\arcosh}{\operatorname{arcosh}}

\newcommand{\eps}{\varepsilon}
\newcommand{\ph}{\varphi}
\newcommand{\cE}{\mathcal E}
\newcommand{\dd}{\,\mathrm d}
\newcommand{\ee}{\mathrm e}

\newcommand{\vol}{\mathrm{Vol}}
\newcommand{\CD}{\mathrm{CD}}
\newcommand{\ones}{\mathbf{1}}

\allowdisplaybreaks[2]
\title[Entropy estimates and volume rigidity]{Entropy estimates and volume rigidity for graphs with positive Bakry-\'Emery curvature}

\author{Yanlong Ding}
\address{\textnormal{Yanlong Ding:}\endgraf\indent School of Mathematical Sciences, University of Science and Technology of China, Hefei 230026, China.}
\email{dylustc@mail.ustc.edu.cn}

\author{Shiping Liu}
\address{\textnormal{Shiping Liu:}\endgraf\indent School of Mathematical Sciences, University of Science and Technology of China, Hefei 230026, China.}
\email{spliu@ustc.edu.cn}

\author{Chiyu Zhou}
\address{\textnormal{Chiyu Zhou:}\endgraf\indent School of Mathematical Sciences, University of Science and Technology of China, Hefei 230026, China.}
\email{dovong@mail.ustc.edu.cn}

\date{September 20, 2026}
\begin{document}
\begin{abstract}
We establish a volume bound for weighted graphs with positive Bakry--\'Emery curvature through entropy estimates for the heat flow and characterize the equality case. As a consequence, if a connected unweighted graph \(G=(V,E)\) satisfies the Bakry--\'Emery curvature--dimension condition \(CD(2,\infty)\) with respect to the unnormalized Laplacian, then
\[
    |V|\le 2^d,
\]
where \(d\) denotes the maximum degree. Equality holds if and only if \(G\) is a hypercube. We also establish an analogous volume rigidity result for the normalized Laplacian.

\end{abstract}
\maketitle

\section{Introduction}\label{sec:introduction}
For a closed Riemannian manifold $(M^n,g)$ with $n\ge 2$ and
$\operatorname{Ric}_M\ge(n-1)g$, classical volume comparison and rigidity give
\[
 \operatorname{Vol}(M)\le \operatorname{Vol}(\mathbb S^n),
\]
with equality if and only if $M$ is isometric to the round unit sphere
\cite{Petersen16}.  In the discrete setting, the hypercube $H_d$ is a natural
analogue of the round sphere.  Hypercube graphs are fundamental structures in
geometry, probability theory, and graph theory
\cite{Gromov99,Bobkov97,HararyHayesWu88}.  For both the unnormalized and
normalized Laplacians, we prove that a graph of maximum degree $d$ satisfying
the Bakry--\'Emery curvature lower bound of $H_d$ has at most $2^d$ vertices,
with equality only for $H_d$.

\subsection{Bakry--\'Emery curvature and \texorpdfstring{$\Gamma$}{Gamma}-calculus}
The $\Gamma$-calculus and curvature-dimension conditions introduced by
Bakry and \'Emery \cite{BE1985} provide a synthetic formulation of lower
Ricci-curvature bounds for Markov semigroups; see \cite{BGL2014} for a
systematic account.  Discrete Bakry--\'Emery theory on graphs was initiated in
\cite{Elworthy91,Schmuck99,LinYau10} and has since yielded graph analogues of
classical manifold results, including gradient estimates \cite{KM2020Elliptic} and stochastic
completeness \cite{HuaLin17}, as well as spectral and
higher-order Buser estimates \cite{LiuPeyerimhoff18}. Recently, Guo, Huang,
and Huang proved that connected bounded-degree graphs with nonnegative
Bakry--\'Emery curvature satisfy volume doubling and hence have polynomial
volume growth \cite{GuoHuangHuang26}.

We next turn to the global consequences of positive Bakry--\'Emery curvature
for graphs. A discrete Bonnet--Myers theorem bounds the diameter and implies
finiteness \cite{LMP2018}. The hypercube is a distinguished
sharp model: it is the unique graph attaining equality in the diameter bound
for the unnormalized Laplacian \cite{LMP}.  On the manifold side, Cheng's
maximal-diameter theorem states that if
$\operatorname{Ric}_M\ge(n-1)g$ and $\operatorname{diam}(M)=\pi$, then $M$
is isometric to $\mathbb S^n$ \cite{Cheng75}.  Together with the volume-rigidity
theorem recalled above, this shows that the round sphere is the common equality
model for both diameter and volume.  This parallel raises the natural discrete
question whether the hypercube is likewise the rigid equality model for both.
More precisely, for the unnormalized or normalized Laplacian, must every graph
of maximum degree $n$ satisfying the curvature bound of $H_n$ have at most
$2^n$ vertices, with equality only for $H_n$? In this article, we answer this
question affirmatively.

A related volume-rigidity result concerns Lin--Lu--Yau curvature. Gamlath,
Liu, Lu, and Yuan proved that a simple connected $\{C_3,C_5\}$-free graph with Lin--Lu--Yau 
curvature on every edge at least $\kappa>0$ satisfies
$|V(G)|\le 2^{2/\kappa}$, with equality exactly when $G\cong H_d$ and
$\kappa=2/d$ \cite{GamlathLiuLuYuan25}.

\subsection{Main results}
Our argument is most naturally formulated in a common weighted setting, from
which the two graph applications will follow. Let $G=(V,b,m)$ be a locally
finite and connected reversible weighted graph, where $b(x,y)=b(y,x)$ are the
edge weights and $m(x)>0$ are the vertex masses, and define the Laplacian and
weighted degree
\[
 Lf(x)=\frac{1}{m(x)}\sum_{y\sim x}b(x,y)(f(y)-f(x)),\qquad
 \Deg(x)=\frac{1}{m(x)}\sum_{y\sim x}b(x,y).
\]
Let $d_x$ denote the combinatorial degree. Define the volume $\vol(G)$ of $G$ by $\sum_{x\in V}m(x)$.

For a probability measure $p$ on a finite vertex set $V$, let
\[
 H(p)=-\sum_{x\in V}p(x)\log p(x)
\]
denote its entropy. For $a>0$, define
\[
 F(a)=\frac a2\int_0^1
 \frac{\log\!\bigl(a(1-s^2)\bigr)}
      {a(1-s^2)-1}\dd s.
\]

Our weighted volume comparison theorem is as follows.

\begin{theorem}\label{thm:weighted-entropy}
Let $G=(V,b,m)$ be a finite connected reversible weighted graph with
$|V|\ge2$. Set $D:=\max_{x\in V}\Deg(x)$,
$d:=\max_{x\in V}d_x$, and $m_{\max}:=\max_{x\in V}m(x)$.
Let $\pi$ be the probability measure on $V$ given by
$\pi(x)=m(x)/\vol(G)$. If $G$ satisfies $\CD(K,\infty)$
for some $K>0$, then
\begin{equation}\label{eq:weighted-comparison}
 \log\frac{\vol(G)}{m_{\max}}
 \le H(\pi)
 \le \frac{2D}{K}F\left(\frac{Kd}{2D}\right).
\end{equation}
\end{theorem}
In particular, when $D=Kd/2$, the quantity
$\frac{2D}{K}F\left(\frac{Kd}{2D}\right)=d\log 2$ is the stationary entropy of
the $d$-dimensional hypercube with constant vertex mass $m_0>0$ and edge
weight $m_0K/2$, so the theorem compares $H(\pi)$ with the entropy of this
weighted hypercube.

We emphasize that the normalization by \(m_{\max}\) in the preceding theorem
is essential. Indeed, replacing \((b,m)\) by \((cb,cm)\) leaves the Laplacian
\(L\) and the curvature condition unchanged, while multiplying both \(\vol(G)\) and
\(m_{\max}\) by \(c\). Thus \(\vol(G)/m_{\max}\) is invariant under this rescaling,
whereas \(\vol(G)\) alone is not.

We next specialize the weighted comparison to the unnormalized and normalized
Laplacians, obtaining the following two volume-rigidity theorems.
\begin{theorem}[Unnormalized cardinality bound and rigidity]\label{thm:unnormalized}
Let $G=(V,E)$ be a finite connected simple unweighted graph, with the unnormalized Laplacian $\Delta_{\mathrm c}$. Put $d=\max_xd_x$ and $\bar d=2|E|/|V|$. If $\CD_{\Delta_{\mathrm c}}(2,\infty)$ holds, then
\begin{equation*}
 |V|\le2^{\bar d}\le2^d.
\end{equation*}
Equality in the first inequality, or equality in the bound
$|V|\le 2^d$, holds if and only if $G$ is a hypercube.
\end{theorem}
\begin{theorem}[Normalized cardinality bound and rigidity]\label{thm:normalized}
Let $G=(V,E)$ be a finite connected simple graph with $|V|\ge2$ and maximum degree $d$. If its normalized Laplacian satisfies $\CD(2/d,\infty)$, then
\begin{equation}\label{eq:normalized-final}
 |V|\le2^d.
\end{equation}
Equality holds if and only if $G\cong H_d$.
\end{theorem}
\subsection{Proof strategy and related work}

The proof is organized around the entropy of heat flow
$H(p_t^x)$ of the heat kernel. The curvature condition allows
us to compare this flow with that of the hypercube and to show that the
hypercube attains the maximum at every time. Its long-time
limit gives the volume bound, while equality in the comparison yields
hypercube rigidity. For background on the information-theoretic ideas used in our proof, see \cite{CoverThomas06,Taneja05}.

Beyond information theory itself, entropy methods have played a major role in
extremal combinatorics. Two landmark examples are Radhakrishnan's entropy proof
of Br\'egman's theorem for permanents \cite{Radhakrishnan97} and Kahn's sharp
bound for independent sets in regular bipartite graphs \cite{Kahn01}. These
results illustrate how entropy can yield sharp bounds in extremal counting,
a viewpoint that also motivates our approach.

Closer to the present geometric setting, probabilistic methods have also been
fruitful for Ollivier--Ricci curvature, from its random-walk and
optimal-transport formulation \cite{Ollivier09} to entropy arguments for
expanders \cite{Salez22} and recent results on volume growth and random-walk
displacement \cite{HutchcroftMunch25,Zhou}.

\subsection*{Use of artificial intelligence}
The authors used generative artificial-intelligence tools to assist in
exploring and refining the entropy-based approach presented in this article.
All mathematical arguments and computations were independently verified by
the authors, who take full responsibility for the content of the article.

\section{Preliminaries}\label{sec:heat}

Let $G=(V,b,m)$ be a connected, locally finite weighted graph. The edge weights
satisfy $b(x,y)=b(y,x)\ge0$ and $b(x,x)=0$, the vertex masses satisfy
$m(x)>0$, and $x\sim y$ means $b(x,y)>0$. Let $d_x$ denote the number of vertices $y$ with $b(x,y)>0$. Let $E$ denote the edge set, that is $ E=\{\{x,y\}:\ b(x,y)>0\}$.

The jump rates, weighted degrees, and Laplacian are
\begin{equation}\label{eq:weighted-setting}
 \begin{aligned}
 q(x,y)&=\frac{b(x,y)}{m(x)},\qquad \Deg(x)=\frac{1}{m(x)}\sum_{y\sim x}b(x,y),\\
 Lf(x)&=\frac{1}{m(x)}\sum_{y\sim x}b(x,y)(f(y)-f(x)).
 \end{aligned}
\end{equation}
The reversibility is encoded by
$m(x)q(x,y)=m(y)q(y,x)$.

For real functions $f,g$, define the Bakry--\'Emery operators
\begin{align*}
 \Gamma(f,g)(x)
 &=\frac12\sum_yq(x,y)(f(y)-f(x))(g(y)-g(x)),\\
 \Gamma_2(f,g)
 &=\frac12\bigl(L\Gamma(f,g)-\Gamma(f,Lg)-\Gamma(g,Lf)\bigr).
\end{align*}
We write $\Gamma(f)=\Gamma(f,f)$ and
$\Gamma_2(f)=\Gamma_2(f,f)$. The Bakry--\'Emery curvature condition
$\CD(K,\infty)$ means that
\[
 \Gamma_2(f)(x)\ge K\Gamma(f)(x)
 \qquad
 \text{for every real-valued function $f$ and every $x\in V$}.
\]
For the results below, we assume
\begin{equation}\label{eq:standing}
 D:=\sup_{x\in V}\Deg(x)<\infty,\qquad
 \CD(K,\infty)\text{ for some }K>0.
\end{equation}

We first recall the global diameter bound and its equality case.

\begin{theorem}[Diameter bound and rigidity]\label{thm:diameter-rigidity}
Let $G=(V,b,m)$ be a connected locally finite weighted graph satisfying
\eqref{eq:standing}. Then
\begin{equation}\label{eq:diameter-bound}
 \diam(G)\le\frac{2D}{K}.
\end{equation}
In particular, $V$ is finite and $D=\max_x\Deg(x)$. If, in addition,
$q(x,y)=q_0$ for every oriented edge and equality holds in
\eqref{eq:diameter-bound}, then $q_0=K/2$, and $G$ is the
$(2D/K)$-dimensional hypercube.
\end{theorem}

The diameter bound follows from \cite[Corollary~2.2]{LMP2018} and the equality case follows from \cite[Theorem~1.4]{LMP}. 

\begin{example}\label{ex:cube-model}
The $n$-dimensional hypercube is the Cartesian product of $n$ copies $K_2$. On $H_n$, take constant mass $m_0>0$ and edge weights $m_0K/2$. This graph satisfies $\CD(K,\infty)$. Its heat kernel is
\begin{equation*}
 P_t(x,z)
 =\left(\frac{1+\ee^{-Kt}}2\right)^n
  \tanh\left(\frac{Kt}{2}\right)^{\dist(x,z)}.
\end{equation*}
For this heat kernel formula, see \cite[Lemma~2.1]{IvanisviliVanHandelVolberg20}. We will give a direct computation in the proof of Lemma~\ref{lem:equality-shells}.
\end{example}

By Theorem~\ref{thm:diameter-rigidity}, $V$ is finite, so we may set
\[
 \vol(G)=\sum_xm(x),\qquad \pi(x)=\frac{m(x)}\vol(G),
\]
and use $d=\max_xd_x$.  All entropy expressions below are therefore on a finite state space; unless stated otherwise, we consider the nontrivial case $\vert V\vert\ge2$.

Write $P_t=\ee^{tL}$ and $p_t^x(z)=P_t\ones_{\{z\}}(x)=P_t(x,z)$. The heat kernel has the following properties:
\begin{enumerate}[label=(\roman*)]
\item \emph{Conservativity and positivity.} For every $x,z\in V$ and $t\ge0$,
\[
 p_t^x(z)\ge0,\qquad \sum_{z\in V}p_t^x(z)=1.
\]
\item \emph{Strict positivity.} Connectedness gives, for every $x,z\in V$ and $t>0$,
\[
 p_t^x(z)>0.
\]
\item \emph{Balance.} Reversibility gives
\begin{equation*}
 m(x)P_t(x,z)=m(z)P_t(z,x).
\end{equation*}
\item \emph{Stationarity.} For every $z\in V$ and $t\ge0$,
\[
 \sum_{x\in V}\pi(x)P_t(x,z)=\pi(z),\qquad \pi P_t=\pi.
\]
\item \emph{Equilibrium limit.} For every $x,z\in V$,
\begin{equation}\label{eq:equilibrium-limit}
 P_t(x,z)\to\pi(z)\qquad(t\to\infty).
\end{equation}
\end{enumerate}

The following gradient estimate and reverse Poincaré inequality follow from \cite[Proposition~2.1]{Bakry97}. For related results on graph Laplacians, see \cite{KM2020Elliptic,LinLiu18,LiuNorbert18}.
\begin{lemma}[Gradient estimate and reverse Poincaré inequality ]\label{lem:reverse}
Under \eqref{eq:standing}, for every bounded real function $f$ and $t>0$,
\begin{align}
 \Gamma(P_tf)&\le\ee^{-2Kt}P_t\Gamma(f),\notag\\
 \Gamma(P_tf)&\le\frac{K}{\ee^{2Kt}-1}
 \bigl(P_t(f^2)-(P_tf)^2\bigr).\label{eq:reverse-bound}
\end{align}
\end{lemma}

\section{Entropies and  their evolution}\label{sec:entropy}

Let $q$ be a probability measure with full support. For any probability
measure $p$, define its entropy and its relative entropy with respect to $q$
(the Kullback--Leibler divergence) by
\begin{equation*}
 \begin{aligned}
 H(p)&=-\sum_zp(z)\log p(z),\\
 \D(p\Vert q)&=\sum_zp(z)\log\frac{p(z)}{q(z)},
 \end{aligned}
\end{equation*}
By convention, $0\log 0=0$.

The elementary inequality $-\log r\ge1-r$ implies $\D(p\Vert q)\ge0$ by summation; equality holds exactly when $p=q$.
Define the entropy of heat flow by
\begin{equation*}
 H_t(x)=H(p_t^x).
\end{equation*}

\begin{proposition}\label{prop:entropy-endpoints}
The following identities hold:
\begin{enumerate}[label=(\roman*)]
\item $H_0(x)=0$ and  $\D(p_0^x\Vert\pi)=-\log\frac{m(x)}{\vol(G)}$;
\item $H_t(x)\to H(\pi)$ and
$\D(p_t^x\Vert\pi)\to0$, as $t\to\infty$.
\end{enumerate}
For any initial probability law $\mu_0$, put $\mu_t=\mu_0P_t$ and $f_t=\mu_t/\pi$. At positive times,
\begin{equation}\label{eq:relative-dissipation}
 \frac{\dd}{\dd t}\D(\mu_t\Vert\pi)
 =-\frac{1}{\vol(G)}\sum_{\{x,y\}\in E}b(x,y)
 (f_t(y)-f_t(x))(\log f_t(y)-\log f_t(x))\le0.
\end{equation}
Consequently each $t\mapsto\D(p_t^x\Vert\pi)$ is nonincreasing.
\end{proposition}
\begin{proof}
The endpoint assertions follow from \eqref{eq:equilibrium-limit} and the continuity of $-s\log s$ on $[0,1]$. 

Reversibility gives $f_t=P_tf_0$, and $f_t>0$ for $t>0$. Since $\partial_tf_t=Lf_t$, differentiation yields
\[
 \frac{\dd}{\dd t}\sum_x\pi(x)f_t(x)\log f_t(x)
 =\sum_x\pi(x)Lf_t(x)+\sum_x\pi(x)\log f_t(x)Lf_t(x).
\]
Stationarity of \(\pi\) gives
\[
 \sum_x\pi(x)Lf_t(x)=0.
\]
Pairing the two orientations of every edge in the remaining term gives \eqref{eq:relative-dissipation}. The summands before the minus sign are nonnegative because logarithm is increasing. Finally take $\mu_0=\ones_x$ to obtain the claimed monotonicity.
\end{proof}

For positive probability measures, define the symmetric relative entropy
\begin{equation*}
 J(p,q):=\D(p\Vert q)+\D(q\Vert p).
\end{equation*}

\begin{proposition}\label{prop:entropy-production}
For every $t>0$,
\begin{equation}\label{eq:entropy-production}
 \frac{\dd}{\dd t}\bigl(\E_\pi H_t\bigr)
 =\frac{1}{\vol(G)}\sum_{\{x,y\}\in E}b(x,y)J(p_t^x,p_t^y)\ge0.
\end{equation}
In particular, $t\mapsto\E_\pi H_t$ increases from $0$ to $H(\pi)$.
\end{proposition}
\begin{proof}
The heat equation in the variable $x$ is
\[
 \partial_tp_t^x(z)=\sum_yq(x,y)(p_t^y(z)-p_t^x(z)).
\]
Since $\sum_zp_t^x(z)=1$, we have
\begin{align*}
 \frac{\dd}{\dd t}\bigl(\E_\pi H_t\bigr)
 &=-\sum_{x,z}\pi(x)\bigl(1+\log p_t^x(z)\bigr)
       \partial_tp_t^x(z)\\
 &=-\sum_x\pi(x)\sum_z\partial_tp_t^x(z)
   -\sum_{x,z}\pi(x)\log p_t^x(z)
       \sum_yq(x,y)\bigl(p_t^y(z)-p_t^x(z)\bigr)\\
 &=\sum_{x,y,z}\pi(x)q(x,y)
       \bigl(p_t^x(z)-p_t^y(z)\bigr)\log p_t^x(z)\\
 &=\frac{1}{\vol(G)}\sum_{\{x,y\}\in E}b(x,y)\sum_z
       \bigl(p_t^x(z)-p_t^y(z)\bigr)
       \bigl(\log p_t^x(z)-\log p_t^y(z)\bigr)\\
 &=\frac{1}{\vol(G)}\sum_{\{x,y\}\in E}b(x,y)
       J\bigl(p_t^x,p_t^y\bigr)\ge0.
\end{align*}
The last equality follows from the definition of \(J\), and the endpoint assertions follow from Proposition~\ref{prop:entropy-endpoints}.
\end{proof}

The rest of the argument bounds the right side of \eqref{eq:entropy-production} by a function with finite integral. That integral limits the equilibrium entropy. 

\section{Nonlinear comparison estimates from curvature}\label{sec:divergences}

For positive probability measures, the chi-square divergence
\cite[Equation~(17.35)]{CoverThomas06} is
\begin{equation*}
 \chi^2(p\Vert q)=\sum_z\frac{(p(z)-q(z))^2}{q(z)}.
\end{equation*}

\begin{proposition}\label{prop:chi-budget}
Under \eqref{eq:standing}, for an oriented edge $(x,y)$ and $t>0$,
\begin{equation}\label{eq:chi-budget}
 \chi^2(p_t^y\Vert p_t^x)
 \le\frac{2K}{q(x,y)(\ee^{2Kt}-1)}.
\end{equation}
\end{proposition}
\begin{proof}
$q(x,y)(P_tf(y)-P_tf(x))^2\leq 2\Gamma(P_tf)(x)$ and \eqref{eq:reverse-bound} give
\[
 \left[\sum_z(p_t^y(z)-p_t^x(z))f(z)\right]^2
 \le\frac{2K}{q(x,y)(\ee^{2Kt}-1)}\Var_{p_t^x}(f).
\]
At this fixed edge and time choose $f(z)=p_t^y(z)/p_t^x(z)-1$. This is a well-defined function by positivity. Set $X=\chi^2(p_t^y\Vert p_t^x)$. Then
\[
 \E_{p_t^x}f=0,\qquad \Var_{p_t^x}(f)=X,\qquad
 \E_{p_t^y}f-\E_{p_t^x}f=X.
\]
Thus $X^2\le 2KX/[q(x,y)(\ee^{2Kt}-1)]$, which proves the result, including $X=0$. 
\end{proof}

\begin{lemma}\label{lem:contrast}
For positive probability measures $p,q$, put
\begin{equation}\label{eq:contrast}
 \nu=\frac{p+q}{2},\qquad u=\frac{p-q}{p+q},\qquad
 W=\frac{u^2}{1-u^2},\qquad
 \ph(v)=\sqrt{\frac{v}{1+v}}\,
 \artanh\sqrt{\frac{v}{1+v}}.
\end{equation}
Then $\E_\nu u=0$ and
\begin{equation}\label{eq:contrast-identities}
 \begin{aligned}
 &\chi^2(p\Vert q)=4\E_\nu[(1+u)W],\quad&
 &\chi^2(q\Vert p)=4\E_\nu[(1-u)W],\\
 &\chi^2(p\Vert q)+\chi^2(q\Vert p)=8\E_\nu W,\quad&
 &J(p,q)=4\E_\nu\ph(W).
 \end{aligned}
\end{equation}
Moreover, $\ph$ is strictly increasing and strictly concave
on $[0,\infty)$.
\end{lemma}
\begin{proof}
The identities follow by direct computation. For the properties of $\ph$, the
integral representation
\begin{equation}\label{eq:phi-integral}
 \ph(v)=\int_0^1\frac{v}{1+(1-s^2)v}\dd s
\end{equation}
is obtained by integrating $(1-r^2s^2)^{-1}$ with $r=\sqrt{v/(1+v)}$. Therefore
\[
 \ph'(v)=\int_0^1\frac{\dd s}{[1+(1-s^2)v]^2}>0,\qquad
 \ph''(v)=-2\int_0^1\frac{1-s^2}{[1+(1-s^2)v]^3}\dd s<0.
\]
\end{proof}

\begin{proposition}[Geometric-mean refinement]\label{prop:geometric-chi}
For $S=\chi^2(p\Vert q)$ and $T=\chi^2(q\Vert p)$,
\begin{equation*}
 J(p,q)\le4\ph\left(\frac{\sqrt{ST}}4\right).
\end{equation*}
\end{proposition}
\begin{proof}
The case $p=q$ is immediate. Otherwise $S,T>0$. For $m_s=(1-s)q+sp$ and $\delta_f=\E_pf-\E_qf$, Cauchy--Schwarz gives
\begin{equation}
      \label{eq:Centered-cauchy-schwarz}
      \Var_q(f)\ge\frac{\delta_f^2}{S},\qquad
 \Var_p(f)\ge\frac{\delta_f^2}{T}.
\end{equation}
The variance of $m_s$ is exactly
\begin{equation}
      \label{eq:Variance-ms}
       \Var_{m_s}(f)=(1-s)\Var_q(f)+s\Var_p(f)+s(1-s)\delta_f^2.
\end{equation}
Put $R_s=\sum_z(p(z)-q(z))^2/m_s(z)$ and choose $f=(p-q)/m_s$. The $m_s$-mean of $f$ is zero, its $m_s$-variance is $R_s$, and $\delta_f=R_s$. Applying \eqref{eq:Variance-ms} to $f$ and \eqref{eq:Centered-cauchy-schwarz} give
\begin{equation*}
 R_s\le\left(\frac{1-s}{S}+\frac{s}{T}+s(1-s)\right)^{-1}.
\end{equation*}

Let $g=\sqrt{ST}$ and $z=g/2+(S+T)/(2g)\geq g/2+1$.
Then we have
\[
 J(p,q)=\int_0^1R_s\dd s\leq \int_0^1\frac{\dd s}{(1-s)/S+s/T+s(1-s)}
 =g\arcosh(z)/\sqrt{z^2-1}.
\]
Since $h(v):=v/\sinh v$ is strictly decreasing on $(0,\infty)$, the function
$z\mapsto h(\arcosh z)=\arcosh(z)/\sqrt{z^2-1}$ is strictly decreasing on $(1,\infty)$. Hence the integral is at most
\[
 g\frac{\arcosh(1+g/2)}{\sqrt{(1+g/2)^2-1}}=4\ph(g/4).
\]
The last identity follows from $\arcosh(1+2a)=2\artanh\sqrt{a/(1+a)}$.
\end{proof}

We introduce the scalar function used in the volume comparison. For $a>0$,
define
\begin{equation}\label{eq:F}
 F(a)=\int_0^\infty\ph\left(\frac{a}{\ee^{2t}-1}\right)\dd t.
\end{equation}
The growth estimates
\begin{equation}\label{eq:phi-growth}
 \ph(v)\le v,\qquad
 \ph(v)\le\frac12\log\bigl(4(1+v)\bigr)
\end{equation}
explain the convergence of the integral defining $F$. The first follows from
\eqref{eq:phi-integral}. For the second, since
$0\le\sqrt{v/(1+v)}<1$,
\[
 \ph(v)
 =\frac12\sqrt{\frac{v}{1+v}}
   \log\frac{1+\sqrt{v/(1+v)}}{1-\sqrt{v/(1+v)}}
 \le\frac12
   \log\frac{1+\sqrt{v/(1+v)}}{1-\sqrt{v/(1+v)}}.
\]
Moreover $1-\sqrt{\frac{v}{1+v}}\geq 1/(2(1+v))$, while $1+\sqrt{v/(1+v)}\le2$. Consequently
\[
 \frac{1+\sqrt{v/(1+v)}}{1-\sqrt{v/(1+v)}}\le4(1+v),
\]
which proves the second estimate in \eqref{eq:phi-growth}. To see the
resulting convergence, note that $a/(\ee^{2t}-1)\sim a/(2t)$ as $t\downarrow0$, so the integrand is
$O(\log(1/t))$ near zero, whereas \eqref{eq:phi-growth} gives
$O(\ee^{-2t})$ at infinity. Thus the integral is finite and positive.

\begin{lemma}\label{lem:F}
The function $F$ defined in \eqref{eq:F} has the following properties.
\begin{enumerate}[label=(\roman*)]
\item It admits the representation
\begin{equation}\label{eq:F-log-representation}
 F(a)=\frac a2\int_0^1
 \frac{\log(a(1-s^2))}{a(1-s^2)-1}\dd s.
\end{equation}
The quotient is understood by continuous extension at $a(1-s^2)=1$;

\item The function $F$ is strictly increasing and strictly concave, and
\begin{equation}\label{eq:F-properties}
 F(1)=\log 2,\qquad
 F(\sqrt{ab})\le\frac{F(a)+F(b)}2;
\end{equation}

\item For fixed $K>0$, the function
\begin{equation}\label{eq:comparison-monotonicity}
 \mathcal C_K(\alpha,\beta)=\frac{2\beta}{K}
 F\left(\frac{K\alpha}{2\beta}\right)
 =2\beta\int_0^\infty\ph\left(\frac{K\alpha}{2\beta(\ee^{2Kt}-1)}\right)\dd t
\end{equation}
is strictly increasing separately in $\alpha,\beta>0$.
\end{enumerate}
\end{lemma}
\begin{proof}
For (i), substituting \eqref{eq:phi-integral} into the definition of $F$ and
applying Tonelli's theorem to the nonnegative integrand give
\[
 \begin{aligned}
 F(a)
 &=\int_0^1\int_0^\infty
   \frac{a}{\ee^{2t}-1+a(1-s^2)}\dd t\dd s.
 \end{aligned}
\]
For fixed $s<1$, put $b=a(1-s^2)>0$. The change of variables
$y=\ee^{-2t}$ gives
\[
 \int_0^\infty\frac{a}{\ee^{2t}-1+b}\dd t
 =\frac a2\int_0^1\frac{\dd y}{1+(b-1)y}
 =\frac a2\frac{\log b}{b-1}.
\]
At $b=1$, the last quotient is understood by continuity as $1$. Therefore
\[
 F(a)=\frac a2\int_0^1
 \frac{\log(a(1-s^2))}{a(1-s^2)-1}\dd s.
\]
This proves \eqref{eq:F-log-representation}. 

For (ii), for every $t>0$, the map
$a\mapsto\ph(a/(\ee^{2t}-1))$ is strictly increasing and strictly concave.
Integrating its strict pointwise inequalities shows that $F$ has the same two
properties. At $a=1$, the integrand of $F$ is
$\ee^{-t}\artanh(\ee^{-t})$, so
\[
 F(1)=\int_0^1\artanh r\dd r=\log2.
\]

Substitution $v=a/(\ee^{2t}-1)$ and integration by parts yield
\begin{equation*}
 F(a)=\frac a2\int_0^\infty\frac{\ph(v)}{v(v+a)}\dd v
 =\frac12\int_0^\infty\ph'(v)\log(1+a/v)\dd v.
\end{equation*}
The boundary terms vanish by \eqref{eq:phi-growth}. Since $s\mapsto\log(1+\ee^s/v)$ is convex, $s\mapsto F(\ee^s)$ is convex. Then Jensen inequality gives second assertion of \eqref{eq:F-properties}.

For (iii), the integrand in \eqref{eq:comparison-monotonicity} increases strictly with $\alpha$. For fixed $\alpha$, differentiating $\beta\ph(c/\beta)$ gives $\ph(v)-v\ph'(v)>0$ at $v=c/\beta>0$: this expression is zero at $v=0$ and has derivative $-v\ph''(v)>0$. Pointwise comparison of finite integrals proves the remaining monotonicity.
\end{proof}

\section{ Weighted total-volume comparison}\label{sec:weighted-bound}

Let $d_\pi=\sum_x\pi(x)d_x$, $D_\pi=\sum_x\pi(x)\Deg(x)$ and set
\begin{equation*}
 a=\frac{Kd_\pi}{2D_\pi},\qquad
 b_t=\frac{a}{\ee^{2Kt}-1},\qquad
 G_*(t)=2D_\pi\ph(b_t).
\end{equation*}

\begin{proposition}\label{prop:joint}
There is a probability measure $\nu_t$ and a nonnegative random variable $W_t$ on $E\times V$ such that $\E_{\nu_t}W_t\le b_t$.Then
\begin{equation}\label{eq:joint-difference}
 \begin{aligned}
 &G_*(t)-\frac{\dd}{\dd t}\bigl(\E_\pi H_t\bigr)\\
 =&2D_\pi\bigl[\ph(b_t)-\ph(\E_{\nu_t}W_t)\bigr]+2D_\pi\bigl[\ph(\E_{\nu_t}W_t)-\E_{\nu_t}\ph(W_t)\bigr]\geq 0,
 \end{aligned}
\end{equation}
with both brackets nonnegative.
\end{proposition}
\begin{proof}
For each edge $e=\{x,y\}$, apply \eqref{eq:contrast} to $p_t^x$ and $p_t^y$, and denote the resulting probability measure and random variable by $\nu_{e,t}$ and $W_{e,t}$, respectively. Define
\[
 \nu_t(e,z)=\frac{2b(x,y)}{\vol(G)D_\pi}\nu_{e,t}(z),\qquad
 W_t(e,z)=W_{e,t}(z).
\]
This is a probability law and gives positive mass to every $(e,z)$. Applying \eqref{eq:chi-budget} to $(x,y)$ and $(y,x)$, and then using \eqref{eq:contrast-identities}, gives
\[
 \E_{\nu_{e,t}}W_{e,t}
 \le\frac{K}{4(\ee^{2Kt}-1)}
 \left(\frac1{q(x,y)}+\frac1{q(y,x)}\right).
\]
Consequently
\begin{equation}\label{eq:joint-average-bound}
 \E_{\nu_t}W_t
 \le\frac{K}{2D_\pi(\ee^{2Kt}-1)}
 \frac{1}{\vol(G)}\sum_{\{x,y\}\in E}b(x,y)
 \left(\frac1{q(x,y)}+\frac1{q(y,x)}\right)
 =\frac{Kd_\pi}{2D_\pi(\ee^{2Kt}-1)}.
\end{equation}
The last equality uses $\frac{b(x,y)}{\vol(G)q(x,y)}=\pi(x)$ and counts incidences. Equations
\eqref{eq:contrast-identities} and \eqref{eq:entropy-production} give
\[
 \frac{\dd}{\dd t}\bigl(\E_\pi H_t\bigr)
 =2D_\pi\E_{\nu_t}\ph(W_t).
\]
By this identity and $G_*(t)=2D_\pi\ph(b_t)$, the left-hand side of
\eqref{eq:joint-difference} equals
$2D_\pi[\ph(b_t)-\E_{\nu_t}\ph(W_t)]$.  Its first bracket is
nonnegative by the monotonicity of $\ph$ and the bound
$\E_{\nu_t}W_t\le b_t$, while its second
bracket is nonnegative by Jensen's inequality and the concavity of $\ph$.
\end{proof}

Now we are ready to give the proof of our main theorem.
\begin{theorem}\label{thm:maintheorem}
Suppose that \(G\) has at least two vertices and satisfies \eqref{eq:standing}. Then
\begin{equation}\label{eq:weighted-comparison}
 \log\frac{M}{m_{\max}}
 \le H(\pi)
 \le \frac{2D_\pi}{K}F\left(\frac{Kd_\pi}{2D_\pi}\right)
 \le \frac{2D}{K}F\left(\frac{Kd}{2D}\right).
\end{equation}
\end{theorem}
\begin{proof}
By \eqref{eq:phi-growth}, $G_*$ is integrable at zero and infinity. Integrate the inequality given by \eqref{eq:joint-difference} over $[\eps,T]$, then use Proposition~\ref{prop:entropy-endpoints} to pass to $\eps\downarrow0$ and $T\to\infty$. Lemma~\ref{lem:F} evaluates the integral $\int_0^\infty G_*(t)\mathrm{d}t$ and gives the exact difference identity
\begin{equation}\label{eq:integrated-difference}
 \frac{2D_\pi}{K}F\left(\frac{Kd_\pi}{2D_\pi}\right)-H(\pi)
 =\int_0^\infty\left(G_*(t)-\frac{\dd}{\dd t}\bigl(\E_\pi H_t\bigr)\right)\dd t.
\end{equation}
The first inequality in \eqref{eq:weighted-comparison} follows from
$H(\pi)=\log \vol(G)-\E_\pi\log m\ge\log \vol(G)-\log m_{\max}$.
The last one follows from $d_\pi\le d$, $D_\pi\le D$, and \eqref{eq:comparison-monotonicity}.
\end{proof}

The dependence on the combinatorial degree $d$ in
\eqref{eq:weighted-comparison} is essential: in general, it cannot be replaced by a bound involving only $K$ and $D$. The following family shows
this directly.
\begin{example}
For $n\ge2$, give $K_n$ constant mass $1$ and rate $q_0=4/(n-1)$
on every oriented edge. Then $\Deg\equiv4$. To see that
$(K_n,q_0,1)$ also satisfies $\CD(2,\infty)$, fix $x$,
write $v_y=f(y)-f(x)$, $S_2=\sum_{y\ne x}v_y^2$, and
$s=\sum_{y\ne x}v_y$. A direct calculation gives
\[
 \Gamma_2(f)(x)
 =\frac{q_0(n+2)}2\Gamma(f)(x)
   +\frac{q_0^2}{2}\bigl((n-1)S_2-s^2\bigr).
\]
The last term is nonnegative by Cauchy--Schwarz, so the curvature is at least
$q_0(n+2)/2=2(n+2)/(n-1)>2$. Here $D=4$ remains fixed, whereas
$d=n-1$ and $\vol(G)/m_{\max}=n$ both grow. Thus no upper bound for the normalized
mass can depend only on $K$ and $D$, and $d$ cannot be omitted from
\eqref{eq:weighted-comparison}. For instance, the bound
$\vol(G)/m_{\max}\le2^{2D/K}$ already fails for $n=17$, $K=2$, and $D=4$.
\end{example}

\begin{proposition}\label{prop:edge-entropy}
Under \eqref{eq:standing},
\begin{equation}\label{eq:edge-entropy}
 H(\pi)\le \frac4K\frac{1}{\vol(G)}\sum_{\{x,y\}\in E}b(x,y)
 F\left(\frac{K}{2\sqrt{q(x,y)q(y,x)}}\right).
\end{equation}
\end{proposition}
\begin{proof}
Proposition~\ref{prop:chi-budget} and Proposition~\ref{prop:geometric-chi} give
\[
 J(p_t^x,p_t^y)\le4\ph\left(
 \frac{K}{2\sqrt{q(x,y)q(y,x)}(\ee^{2Kt}-1)}\right).
\]
Multiply by $b(x,y)$, sum over the edges, divide by $\vol(G)$, and integrate using \eqref{eq:entropy-production}. The change of variables $u=Kt$ in the definition of $F$ gives \eqref{eq:edge-entropy}. 
\end{proof}

\section{The weighted volume rigidity}\label{sec:equality}

We identify equality in the weighted entropy estimate of
Theorem~\ref{thm:weighted-entropy} before applying it to an
unweighted graph. This also gives the weighted rigidity theorem. In the equality case, we have
\begin{equation*}
 H(\pi)=\frac{2D_\pi}{K}F(a),\qquad
 a=\frac{Kd_\pi}{2D_\pi}.
\end{equation*}

\begin{lemma}\label{lem:equality-ratios}
Every oriented edge has the same rate $q_0=K/(2a)$, and $m$ is constant. With
\begin{equation}\label{eq:tau-general}
 r(t)=\sqrt{\frac{a}{\ee^{2Kt}-1+a}},\qquad
 \tau(t)=\frac{1-r(t)}{1+r(t)},
\end{equation}
for every $x\sim y$, $z\in V$, and $t>0$, the ratio
$P_t(x,z)/P_t(y,z)$ is either $\tau(t)$ or $\tau(t)^{-1}$.
\end{lemma}
\begin{proof}
The continuous nonnegative integrand in \eqref{eq:integrated-difference} has
integral zero. It vanishes at every positive time, so both brackets in
\eqref{eq:joint-difference} vanish. Strict monotonicity of $\ph$ gives
$\E_{\nu_t}W_t=b_t$ from the first bracket. Equality in Jensen's inequality
for the strictly concave function $\ph$ then gives $W_t=b_t$
$\nu_t$-almost surely. Thus $W_{e,t}(z)=b_t$ for every $(e,z)$, since all
these pairs have positive $\nu_t$-mass.

For each edge, $\E_{\nu_{e,t}}u=0$ and \eqref{eq:contrast-identities} now give both $ \chi^2(p_t^y\Vert p_t^x)$ and  $\chi^2(p_t^x\Vert p_t^y)$ equal to $4b_t$. Equality in \eqref{eq:joint-average-bound} also forces equality in each individual estimate, since the difference in each estimate is nonnegative and has a positive coefficient. Therefore
\[
 \frac{4a}{\ee^{2Kt}-1}
 =\frac{2K}{q(x,y)(\ee^{2Kt}-1)},
\]
which gives $q(x,y)=q_0$ in every orientation. Symmetry of $b$ gives $m(x)q_0=m(y)q_0$ along edges; connectedness makes $m$ constant.

The identity $W=b_t$ gives $|u|=r(t)$ and hence the ratio is either $\tau(t)$ or $\tau(t)^{-1}$. Since each ratio is continuous in $t$ and $0<\tau(t)<1<\tau(t)^{-1}$, for each fixed $(x,y,z)$ it is either always $\tau(t)$ or always $\tau(t)^{-1}$.
\end{proof}

\begin{lemma}\label{lem:equality-scale}
The graph taking equality in Theorem~\ref{thm:maintheorem} is bipartite and $n$-regular for an integer $n\ge1$. Moreover,
\begin{equation}\label{eq:equality-scale}
 q_0=K/2,\qquad a=1,\qquad \tau(t)=\tanh(Kt/2).
\end{equation}
\end{lemma}
\begin{proof}
By Lemma~\ref{lem:equality-ratios}, $L=q_0\Delta_{\mathrm c}$ and the heat kernel is symmetric. If $\dist(v,z)=j$, matrix expansion gives
\begin{equation}\label{eq:small-time}
 P_t(v,z)=\frac{q_0^j g(v,z)}{j!}t^j+O(t^{j+1}),
\end{equation}
where $g(v,z)>0$ is the number of geodesics. 

Fix a root $o$. An edge between two vertices at the same distance from $o$ would make $P_t(x,o)/P_t(y,o)$ tend to a finite positive limit, contradicting the ratio alternative in Lemma~\ref{lem:equality-ratios} because $\tau(t)\to0$. Thus every edge joins consecutive distance layers, and the graph is bipartite.

For an edge $x\sim y$, since both $P_t(x,y)$ and $P_t(y,x)$ are $O(t)$,  $P_t(y,x)=\tau(t)P_t(x,x)$ and $P_t(x,y)=\tau(t)P_t(y,y)$ for all positive times. Symmetry implies $P_t(x,x)=P_t(y,y)$. Comparing the coefficients of $t$ in $P_t(x,x)=1-q_0d_xt+O(t^2)$ makes degrees equal along every edge. Connectedness gives regularity, say $d_x=n$.

To determine the rate, expansion of \eqref{eq:tau-general} gives
\[
 \tau(t)=q_0t+q_0(K-2q_0)t^2+O(t^3).
\]
 On the other hand, bipartiteness implies $(A^2)_{xy}=0$ on an edge, where $A$ is adjacency. Since $L=q_0(A-nI)$,
\[
 P_t(y,x)=q_0t-nq_0^2t^2+O(t^3),\qquad
 P_t(x,x)=1-nq_0t+O(t^2).
\]
Their ratio is $q_0t+O(t^3)$. Comparison with $\tau$ forces $K=2q_0$, and then all assertions in \eqref{eq:equality-scale} follow.
\end{proof}

\begin{lemma}\label{lem:equality-shells}
For every root $o$ and every $z$, let $i=\dist(o,z)$. Then $z$ has exactly $i$ neighbors in $S_{i-1}(o)$, none in $S_i(o)$, and $n-i$ in $S_{i+1}(o)$. Moreover, $\max_{z\in V}\dist(o,z)=n$ for every root $o$, and
\begin{equation}\label{eq:cube-spheres}
 |S_i(o)|=\binom ni\qquad(0\le i\le n).
\end{equation}
\end{lemma}
\begin{proof}
Write $q_0=K/2$ and $\tau=\tanh(q_0t)$. For an edge pointing away from $o$, \eqref{eq:small-time} selects the ratio $\tau$. Along a geodesic and by symmetry,
\begin{equation}\label{eq:radial-equality-kernel}
 P_t(o,z)=A_o(t)\tau(t)^{\dist(o,z)},\qquad A_o(t)=P_t(o,o).
\end{equation}
At the root, the heat equation gives $A_o'=nq_0(\tau-1)A_o$ and $A_o(0)=1$. Therefore
\[
 A_o(t)=\bigl(\ee^{-q_0t}\cosh(q_0t)\bigr)^n
 =\left(\frac{1+\ee^{-Kt}}2\right)^n.
\]
Let $c_o(z)$ be the number of inward neighbors of $z\in S_i(o)$. There are no same-layer edges. Differentiating \eqref{eq:radial-equality-kernel} and using $\tau'/\tau=q_0(\tau^{-1}-\tau)$ gives
\[
 \frac{\partial_tP_t(o,z)}{P_t(o,z)}
 =nq_0(\tau-1)+iq_0(\tau^{-1}-\tau).
\]
Applying the heat equation in the variable $z$ instead gives
\[
 \begin{aligned}
 \frac{\partial_tP_t(o,z)}{P_t(o,z)}
 &=\frac{q_0}{P_t(o,z)}
   \sum_{y\sim z}\bigl(P_t(o,y)-P_t(o,z)\bigr)
 &=nq_0(\tau-1)+c_o(z)q_0(\tau^{-1}-\tau).
 \end{aligned}
\]
Comparison with the preceding expression gives
\[
 (i-c_o(z))q_0(\tau^{-1}-\tau)=0,
 \qquad\text{hence}\qquad c_o(z)=i.
\]
At a farthest vertex every neighbor is inward, so its distance is $n$. Counting edges between consecutive spheres yields
$(n-i)|S_i(o)|=(i+1)|S_{i+1}(o)|$, which proves \eqref{eq:cube-spheres}.
\end{proof}

\begin{theorem}[Weighted entropy equality and rigidity]\label{thm:weighted-rigidity}
For a nontrivial graph satisfying \eqref{eq:standing}, the following are equivalent:
\begin{enumerate}[label=(\roman*)]
\item $H(\pi)=(2D_\pi/K)F(Kd_\pi/(2D_\pi))$;
\item $\log(\vol(G)/m_{\max})=(2D/K)F(Kd/(2D))$;
\item $G=(H_n, KC/2, C)$ for some $n\geq 1$ and positive constant $C$.
\end{enumerate}
In the equality case $n=d=2D/K$.
\end{theorem}
\begin{proof}
Assume (i). Lemmas~\ref{lem:equality-ratios} and \ref{lem:equality-scale} give
$m\equiv C$ for some $C>0$, $q(x,y)=K/2$ on every oriented edge, and
$n$-regularity for some $n\ge1$. Hence  $D=\max_x\Deg(x)=nK/2.$
By Lemma~\ref{lem:equality-shells}, $\max_{z\in V}\dist(o,z)=n$ for every $o\in V$, so
\[
 \diam(G)=n=\frac{2D}{K}.
\]
Theorem~\ref{thm:diameter-rigidity} identifies the graph taking equality in Theorem~\ref{thm:maintheorem} as $H_n$.
Moreover, on every edge,
\[
 b(x,y)=m(x)q(x,y)=\frac{KC}{2},
\]
and therefore $G=(H_n,KC/2,C)$. This proves (iii), as well as
$n=d=2D/K$.

Conversely, (iii) implies equality throughout
\eqref{eq:weighted-comparison} by Example~\ref{ex:cube-model}, and hence implies
both (i) and (ii). Finally, (ii) immediately implies (i).
\end{proof}

The use of $m_{\max}$ is essential for the rigidity conclusion.
Replacing $\vol(G)/m_{\max}$ by $\vol(G)/m(o)$ for a prescribed vertex $o$ can create
equality without a hypercube. The following example is based on the Hamming
projection introduced in \cite[Definition~4.5]{LMP}.
\begin{example}[{\cite[Section~4.3]{LMP}}]
Fix $n\geq2$ and let $P_{n+1}=\{0\sim1\sim\cdots\sim n\}$ be the
weighted quotient of $H_n$, obtained by collapsing the Hamming spheres
\[
 S_i=\{v\in H_n:|v|=i\},\qquad 0\leq i\leq n.
\]
Give this path the masses and edge weights
\[
 m(i)=\binom{n}{i},\qquad
 b(i,i+1)=\binom{n}{i}(n-i),\qquad 0\leq i<n.
\]
Thus $m(i)=|S_i|$, while $b(i,i+1)$ is the number of cube edges joining
$S_i$ to $S_{i+1}$. 

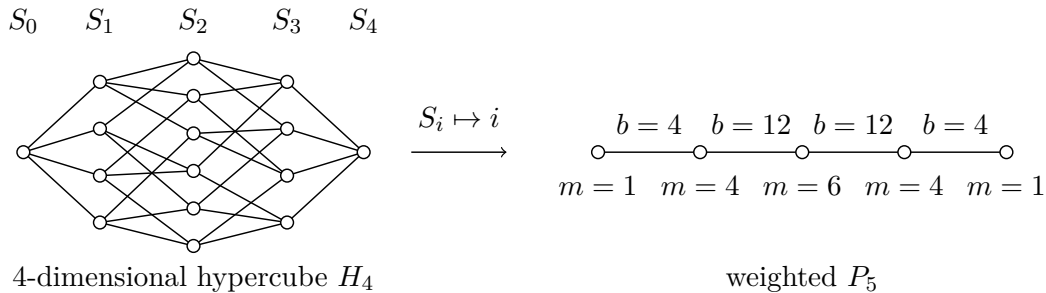
\begin{figure}[htbp]
      \centering
      \input{figures/q4HammingWeightQuotient.tex}
      \caption{The weighted quotient of $H_4$}
\end{figure}

If a function $f$ on $P_{n+1}$ is lifted to $H_n$ by $Rf(v)=f(|v|)$, then
\[
 \Delta_{H_n}Rf=RL_{P_{n+1}}f.
\]
Together with $R(fg)=(Rf)(Rg)$, this identity lifts $\Gamma$ and
$\Gamma_2$ from the quotient to the cube. Hence the $\CD(2,\infty)$ bound on
$H_n$ gives $\CD(2,\infty)$ on this weighted $P_{n+1}$. Moreover,
\[
 \Deg(i)=\frac{b(i-1,i)+b(i,i+1)}{m(i)}=i+(n-i)=n,
\]
with the evident endpoint convention. Thus $D=n$, while $K=2$ and
$\vol(G)=\sum_{i=0}^n\binom{n}{i}=2^n$. Since $m(0)=1$ and
$m_{\max}=\binom{n}{\lfloor n/2\rfloor}$,
\[
 \frac{\vol(G)}{m(0)}=2^n=2^{2D/K},
 \qquad
 \frac{\vol(G)}{m_{\max}}
 =\frac{2^n}{\binom{n}{\lfloor n/2\rfloor}}
 <2^n=2^{2D/K}.
\]
For $n\geq2$, the weighted path $P_{n+1}$ is not isomorphic to $H_n$.
Thus the rooted normalized mass $\vol(G)/m(0)$ attains equality without a
hypercube, whereas the normalization in \eqref{eq:weighted-comparison}
remains strict.
\end{example}

\section{Application to the unnormalized Laplacian}\label{sec:unnormalized}
We now apply the Theorem~\ref{thm:maintheorem} and rigidity theorems to the
unnormalized Laplacian.

\begin{proof}[Proof of Theorem~\ref{thm:unnormalized}]
For $N\ge2$, take $m\equiv1$, $q(x,y)=1$, and $K=2$. Then $\pi=1/N$, $d_\pi=D_\pi=\bar d$, and $a=1$. Theorem~\ref{thm:weighted-entropy} gives
\[
 \log N=H(\pi)\le\bar dF(1)=\bar d\log2\le d\log2.
\]
Equality in the first bound is precisely (i) of Theorem~\ref{thm:weighted-rigidity}, so Theorem~\ref{thm:weighted-rigidity} identifies a cube. Equality in the outer bound $|V|\le2^d$ forces equality in both intermediate bounds. Conversely, a cube satisfies curvature and cardinality equality by Example~\ref{ex:cube-model}. 
\end{proof}

\begin{corollary}\label{cor:unnormalized-time}
Under the nontrivial hypotheses of Theorem~\ref{thm:unnormalized},
\begin{equation*}
 \E_\pi H_t\le\bar d\,h\left(\frac{1-\ee^{-2t}}2\right),
\end{equation*}
where $h(s)=-s\log s-(1-s)\log(1-s)$ for $0\leq s\leq 1$ and
\begin{equation*}
 \bar{d}\log2-\log N
 =
 \int_0^\infty\left[2\bar d\,\ee^{-2t}\artanh(\ee^{-2t})
 -\frac{\dd}{\dd t}\bigl(\E_\pi H_t\bigr)\right]\dd t.
\end{equation*}
The integrand is continuous and nonnegative at positive times.
\end{corollary}
\begin{proof}
Substitute $K=2$, $D_\pi=\bar d$, and $a=1$ into Proposition~\ref{prop:joint}, Theorem~\ref{thm:weighted-entropy}. Since $H(\pi)=\log N$ and $\int_0^t\varphi\left(1/(e^{4s}-1)\right)\mathrm{d}s=h\left(\frac{1-\ee^{-2t}}2\right)$, rearranging \eqref{eq:integrated-difference} gives the displayed identity.
\end{proof}

\section{Application to the normalized Laplacian}\label{sec:normalized}
We now establish the corresponding rigidity theorem for the normalized Laplacian.

Throughout this section, $G$ is a finite connected simple graph with $N\ge2$, and its normalized Laplacian satisfies
\begin{equation}\label{eq:normalized-setting}
 L_{\mathrm n}f(x)=\frac1{d_x}\sum_{y\sim x}(f(y)-f(x)),
 \qquad \CD_{L_{\mathrm n}}(2/d,\infty).
\end{equation}
In the weighted notation of \eqref{eq:weighted-setting}, this corresponds to $m(x)=d_x$ and unit edge weights; hence $q(x,y)=1/d_x$ and $\pi(x)=d_x/\vol(G)$.

Unlike the unnormalized case, its stationary measure depends on the vertex degrees, so passing from entropy to cardinality requires a separate argument.

\begin{proposition}\label{prop:normalized-entropy}
Under \eqref{eq:normalized-setting},
\begin{equation*}
 H(\pi)\le I(G):=\frac{2d}{\vol(G)}\sum_{\{x,y\}\in E}
 F\left(\frac{\sqrt{d_xd_y}}d\right).
\end{equation*}
\end{proposition}
\begin{proof}
Apply Proposition~\ref{prop:edge-entropy} with $K=2/d$, $q(x,y)=1/d_x$, and $b(x,y)=1$. Its argument of $F$ is $\sqrt{d_xd_y}/d$, and $\frac{4b(x,y)}{K\vol(G)}=\frac{2d}{\vol(G)}$. 
\end{proof}

For real $1\leq k,l\leq\delta $ define
\begin{equation*}
 h_\delta(k)=\frac{\delta}{k}-1-\log\frac{\delta}{k},
 \qquad
 \cE_\delta(k,l)=2\delta\left[\log 2-F\left(\frac{\sqrt{kl}}{\delta}\right)\right]
                 -h_\delta(k)-h_\delta(l).
\end{equation*}

\begin{lemma}\label{lem:edge-compensation}
The following inequalities hold:
\begin{equation}\label{eq:edge-compensation}
 \begin{aligned}
 &\cE_\delta(k,l)\ge0 &&(2\le k,l\le\delta),\\
 &\cE_\delta(1,l)>0 &&(3\le l\le\delta).
 \end{aligned}
\end{equation}
In the first line, equality holds exactly when $k=l=\delta$.
\end{lemma}
For the proof of Lemma~\ref{lem:edge-compensation}, see Appendix~\ref{sec:scalar}.
\begin{lemma}\label{lem:counting}
Under \eqref{eq:normalized-setting}, let $\eta=dN/\vol(G)$. Then
\begin{equation}\label{eq:normalized-exact}
 \begin{aligned}
 d\log 2-\log N={}&\frac{1}{\vol(G)}\sum_{\{x,y\}\in E}\cE_d(d_x,d_y)+\bigl(I(G)-H(\pi)\bigr)+\bigl(\eta-1-\log\eta\bigr).
 \end{aligned}
\end{equation}
The last two terms are nonnegative. Also $\eta\ge1$, with equality of the last term exactly when $G$ is $d$-regular.
\end{lemma}
\begin{proof}
Since $\pi(x)=d_x/\vol(G)$ and $\eta=dN/\vol(G)$,
\[
 \log N-H(\pi)
 =\sum_x\pi(x)\log\bigl(N\pi(x)\bigr)
 =\log\eta-\E_\pi\log\frac d{d_x}.
\]
Counting every vertex $x$ once for each of its $d_x$ incident edges gives
\begin{equation}\label{eq:counting-cost}
 \begin{aligned}
 &\frac{1}{\vol(G)}\sum_{\{x,y\}\in E}\bigl(h_d(d_x)+h_d(d_y)\bigr)\\
 =&\E_\pi h_d(d_x)
 =\log N-H(\pi)+\eta-1-\log\eta.
 \end{aligned}
\end{equation}
On the other hand, the definitions of $\cE_d$ and $I(G)$, together with $\vol(G)=2|E|$, yield
\[
 \frac{1}{\vol(G)}\sum_{\{x,y\}\in E}\cE_d(d_x,d_y)
 =d\log2-I(G)
  -\frac{1}{\vol(G)}\sum_{\{x,y\}\in E}\bigl(h_d(d_x)+h_d(d_y)\bigr).
\]
Substituting \eqref{eq:counting-cost} and rearranging proves \eqref{eq:normalized-exact}. Proposition~\ref{prop:normalized-entropy} gives $I(G)-H(\pi)\ge0$, while $\vol(G)\le dN$ gives $\eta\ge1$ and hence $\eta-1-\log\eta\ge0$. Moreover, $\eta=1$ holds exactly when every $d_x=d$.
\end{proof}

By Lemma~\ref{lem:edge-compensation}, if the degree pair $(d_x,d_y)$ with $d_x\leq d_y$ belongs to 
\[
 \{(k,l):2\le k,l\le d\}\cup\{(1,l):3\le l\le d\},
\]
then  $\cE_d(k,l)\ge0$. If the degree pair is $(1,1)$, the only possible graph is $K_2$. Consequently, the only pair not
covered by \eqref{eq:edge-compensation}, and hence to be treated directly from
the curvature condition, is
\[
 \{d_x,d_y\}=\{1,2\}.
\]

\begin{lemma}\label{lem:two-branch}
Let $v$ have exactly two neighbors $u,w$ of degrees $r,b$. Suppose $u,w$ are not adjacent and
$(N(u)\setminus\{v\})\cap(N(w)\setminus\{v\})=\varnothing$.
Prescribe $f(v)=0$, $f(u)=s$, $f(w)=t$, and set $f=2s$ on $N(u)\setminus\{v\}$ and $f=2t$ on $N(w)\setminus\{v\}$. For the normalized Laplacian $L_{\mathrm n}$,
\begin{equation}\label{eq:two-branch}
 \Gamma(f)(v)=\frac{s^2+t^2}{4},\qquad
 8\Gamma_2(f)(v)=\left(\frac4r-1\right)s^2+2st+
                    \left(\frac4b-1\right)t^2.
\end{equation}
\end{lemma}
\begin{proof}
The hypotheses make the assignments consistent; values outside this two-ball do not enter $\Gamma_2(f)(v)$. Direct calculation gives
\[
 L_{\mathrm n}f(v)=\frac{s+t}{2},\quad
 L_{\mathrm n}f(u)=\left(1-\frac2r\right)s,
 \quad L_{\mathrm n}f(w)=\left(1-\frac2b\right)t.
\]
All increments from $u$ have absolute value $|s|$, and those from $w$ have absolute value $|t|$. Thus $\Gamma(f)(u)=s^2/2$, $\Gamma(f)(w)=t^2/2$, and $L_{\mathrm n}\Gamma(f)(v)=0$. Moreover,
\[
 \Gamma(f,L_{\mathrm n}f)(v)=\frac14\left[
 s\left(\left(1-\frac2r\right)s-\frac{s+t}{2}\right)
 +t\left(\left(1-\frac2b\right)t-\frac{s+t}{2}\right)\right].
\]
Substitution in
$\Gamma_2=L_{\mathrm n}\Gamma/2-\Gamma(f,L_{\mathrm n}f)$ proves
\eqref{eq:two-branch}.
\end{proof}

\begin{lemma}\label{lem:leaf-two}
Under \eqref{eq:normalized-setting}, the existence of an edge with degrees $(1,2)$ implies $G=P_3$.
\end{lemma}
\begin{proof}
Let a leaf $u$ be adjacent to a degree-two vertex $v$, whose other neighbor is $w$ of degree $r$. Lemma~\ref{lem:two-branch} applies at $v$ with neighbor degrees $1,r$. If $r\ge3$, choose $s=-1/3$, $t=1$. Then $\Gamma(f)(v)>0$ but
$8\Gamma_2(f)(v)=4/r-4/3\le0$, contradicting the positive curvature lower bound. If $r=1$, connectedness gives $P_3$.

Suppose $r=2$ and let $z\ne v$ be the other neighbor of $w$, with degree $b$.  $v$ already has neighbors $u,w$, and the leaf $u$ cannot be adjacent to $z$. Thus $z\not\sim v$ and $N(v)\setminus\{w\}\cap N(z)\setminus\{w\}=\varnothing$. With neighbor degrees $2,b$ and $s=-1$, $t=1$, \eqref{eq:two-branch} gives $\Gamma(f)(w)>0$ and $8\Gamma_2(f)(w)=4/b-2$. Positivity forces $b=1$, so connectedness gives $G=P_4$ and $d=2$.

But $P_4$ fails the required $\CD(1,\infty)$. At $v$, with neighbor degrees $1,2$ and $s=-1$, $t=1$, formula \eqref{eq:two-branch} gives $\Gamma(f)(v)=1/2$ and $\Gamma_2(f)(v)=1/4$. This contradiction leaves only $P_3$.
\end{proof}

\begin{proof}[Proof of Theorem~\ref{thm:normalized}]
An edge with a degree pair $(1,1)$ forces $K_2$, which has $N=2=2^d$. An edge with a degree pair $(1,2)$ forces $P_3$ by Lemma~\ref{lem:leaf-two}; here $N=3<4=2^d$. Exclude these two graphs. For every remaining edge, Lemma~\ref{lem:edge-compensation} gives $\cE_d(d_x,d_y)\ge0$, so the edge sum in \eqref{eq:normalized-exact} is nonnegative. The two other terms are nonnegative as well. This proves \eqref{eq:normalized-final}.

If equality holds, \eqref{eq:normalized-exact} forces $\eta-1-\log\eta=0$. Since $\eta\ge1$, $\eta=1$, and Lemma~\ref{lem:counting} gives $d_x=d$ at every vertex. Since $G$ is now known to be $d$-regular, the normalized Laplacian
$L_{\mathrm n}$ and the unnormalized Laplacian $\Delta_{\mathrm c}$
are related by
\[
 L_{\mathrm n}=\frac1d\Delta_{\mathrm c},\qquad
 \Gamma_{L_{\mathrm n}}=\frac1d\Gamma_{\Delta_{\mathrm c}},\qquad
 \Gamma_{2,L_{\mathrm n}}=\frac1{d^2}\Gamma_{2,\Delta_{\mathrm c}}.
\]
Thus $\CD_{L_{\mathrm n}}(2/d,\infty)$ implies $\CD_{\Delta_{\mathrm c}}(2,\infty)$. Since $N=2^d$, Theorem~\ref{thm:unnormalized} identifies $H_d$. Conversely, Example~\ref{ex:cube-model} and the same scaling verify $\CD(2/d,\infty)$ on the normalized cube, which has $2^d$ vertices.

\end{proof}

\begin{corollary}\label{cor:normalized-irregularity-bound}
Under Theorem~\ref{thm:normalized}, if $G\notin\{K_2,P_3\}$, then
\begin{equation*}
 \log\frac{2^d}{N}
 \ge\frac{1}{\vol(G)}\sum_{\{x,y\}\in E}\cE_d(d_x,d_y)
       +\frac{dN}{\vol(G)}-1-\log\frac{dN}{\vol(G)}\ge0.
\end{equation*}
Both terms on the right are strictly positive for a nonregular graph.
\end{corollary}
\begin{proof}
Drop only $I(G)-H(\pi)\ge0$ from \eqref{eq:normalized-exact}. The final scalar term is strictly positive when $\vol(G)<dN$. For the first term, choose a vertex $x$ whose $d_x<d$. For any edge $x\sim y$, by the strict part of Lemma~\ref{lem:edge-compensation}, $\cE_d(d_x,d_y)>0$. All other edge contributions are nonnegative.
\end{proof}

\section{Acknowledgments}
This work was supported by the Scientific Research Innovation Capability Support Project for Young Faculty (Grant No. SRICSPYF-ZY2025160) and the National Natural Science Foundation of China (Grant No. 12431004).

\appendix
\section{Proof of Lemma~\ref{lem:edge-compensation}}\label{sec:scalar}

The normalized argument in Section~\ref{sec:normalized} requires quantitative control of the degree-dependent terms that arise when entropy is converted into cardinality. We continue to use the function $F$ from \eqref{eq:F}. The next two lemmas provide an elementary upper bound for $F$ and a strict barrier for the leaf case; together they lead to the edgewise estimate in Lemma~\ref{lem:edge-compensation}.

\begin{lemma}\label{lem:envelope}
For $0<a\le1$,
\begin{equation}\label{eq:new-envelope}
 F(a)\le a\log 2-\frac a2\log a-\frac a5(1-a)
                  -\frac{2a}{45}(1-a)^2.
\end{equation}
In addition,
\begin{equation}\label{eq:F-derivative-bound}
 F(a)-aF'(a)
 \le\frac a2-\frac{a^2}5-\frac{4a^2}{45}(1-a)
 \le\frac a2.
\end{equation}
\end{lemma}
\begin{proof}
Put $c=1-s^2$. Differentiating \eqref{eq:F-log-representation} after division
by $a$ gives
\begin{equation}\label{eq:G-integral}
 \left(\frac{F(a)}a\right)'+\frac1{2a}
 =\frac12\int_0^1c\,g(ac)\dd s,
 \qquad g(u)=\frac{u-1-\log u}{(u-1)^2}.
\end{equation}
All singularities are removable, which justifies the differentiation. The representation
\[
 g(u)=\int_0^1\frac{1-r}{[1+r(u-1)]^2}\dd r
\]
shows that for $0<u\le1$,
\[
 -g'(u)=2\int_0^1\frac{r(1-r)}{[1+r(u-1)]^3}\dd r\ge\frac13.
\]
Consequently $g(ac)\ge g(c)+(1-a)c/3$ when $a\le1$. At $a=1$, expansion of $-\log(1-s^2)$ and Tonelli give the simple lower bound
\[
 \frac12\int_0^1c\,g(c)\dd s
 =\sum_{j=0}^\infty\frac1{(j+2)(2j+1)(2j+3)}
 \ge\sum_{j=0}^4\frac1{(j+2)(2j+1)(2j+3)}
 =\frac{8353}{41580}>\frac15.
\]
Since $\int_0^1(1-s^2)^2\dd s=8/15$, \eqref{eq:G-integral} yields
\begin{equation}\label{eq:F-ratio-derivative}
 \left(\frac{F(a)}a\right)'+\frac1{2a}
 \ge\frac15+\frac4{45}(1-a).
\end{equation}
Integrate from $a$ to $1$ and use $F(1)=\log 2$ to obtain
\eqref{eq:new-envelope}. Finally,
$F-aF'=-a^2(F/a)'$, so \eqref{eq:F-ratio-derivative} also proves
\eqref{eq:F-derivative-bound}. 
\end{proof}

\begin{lemma}\label{lem:leaf-barrier}
For every $0<a<1$,
\begin{equation*}
 \mathscr R(a):=\log 2-\frac12+\frac{a^2}{2}-a^2\log a-F(a)>0.
\end{equation*}
At $a=1$ the expression is zero.
\end{lemma}
\begin{proof}
We use only
\begin{equation}\label{eq:coarse-logs}
 \frac{69}{100}<\log 2<\frac7{10},\qquad \log3>\frac{13}{12}.
\end{equation}
Indeed $\log((1+z)/(1-z))=2\sum_{j\ge0}z^{2j+1}/(2j+1)$. The first two terms at $z=1/3$ give $\log 2>56/81>69/100$, while bounding all $1/(2j+1)$ with $j\geq 1$ by $1/3$ gives $\log 2<25/36<7/10$. The first two terms at $z=1/2$ give the bound for $\log3$.

By \eqref{eq:new-envelope}, $\mathscr R(a)\ge U(a)$, where
\begin{equation}\label{eq:U-function}
 U(a)=(1-a)\left(\log 2-\frac12-\frac{3a}{10}\right)
       +a\left(\frac12-a\right)\log a
       +\frac{2a}{45}(1-a)^2.
\end{equation}
First suppose $0<a\le1/3$. Direct differentiation gives
\[
 U''(a)=\frac{4a}{15}-2\log a-\frac{116}{45}+\frac1{2a},
 \qquad
 U'''(a)=\frac4{15}-\frac2a-\frac1{2a^2}<0.
\]
Hence on this interval
\[
 U''(a)\ge U''(1/3)=-\frac{89}{90}+2\log3>1.
\]
At $a_0=1/4$, the same elementary calculation gives
\[
 U(a_0)=\frac58\log 2-\frac{17}{40}>\frac1{160},
 \qquad U'(a_0)=\frac{73}{120}-\log 2,\qquad |U'(a_0)|<\frac1{10}.
\]
Taylor's formula with the lower bound $U''\ge1$ now gives uniformly for $0<a\le1/3$,
\[
 U(a)\ge U(a_0)+U'(a_0)(a-a_0)+\frac12(a-a_0)^2
 \ge U(a_0)-\frac12U'(a_0)^2>\frac1{800}.
\]
For $1/3\le a\le1/2$, $U'''<0$ says that $U'$ is concave. Its two endpoint values are positive:
\[
 U'(1/3)=\frac{17}{30}-\log 2+\frac16\log3>\frac{17}{360},
 \qquad U'(1/2)=\frac{22}{45}-\frac{\log 2}{2}>0.
\]

For $a\in[1/3,1/2]$, concavity gives
$
 U'(a)\ge \min\bigl\{U'(1/3),U'(1/2)\bigr\}>0.
$

Finally suppose $1/2\le a<1$. Discard the last nonnegative term in \eqref{eq:U-function}. Since $a(1/2-a)\le0$ and $\log a\le a-1$,
\[
 U(a)\ge(1-a)\left(\log 2-\frac12-\frac45a+a^2\right)
 \ge(1-a)\left(\log 2-\frac{13}{20}\right)>0.
\]
The middle inequality holds because the quadratic is increasing for $a\ge1/2$; the final one follows from \eqref{eq:coarse-logs}. This covers the whole interval $(0,1)$, and $F(1)=\log 2$ gives the endpoint statement.
\end{proof}

\begin{proof}[Proof of Lemma~\ref{lem:edge-compensation}]
For $2\le k\le\delta$ put $a=k/\delta$. Monotonicity gives $\log 2-F(a)\ge0$. Discarding the last two negative terms in \eqref{eq:new-envelope},
\[
 \begin{aligned}
 k[\log 2-F(a)]-[1-a+a\log a]
 &\ge2[\log 2-F(a)]-[1-a+a\log a]\\
 &\ge(2\log 2-1)(1-a)\ge0.
 \end{aligned}
\]
Divide by $a$ to obtain $\delta[\log 2-F(k/\delta)]\ge h_\delta(k)$, strictly if $k<\delta$. The second equation in \eqref{eq:F-properties} then gives
\[
 2\delta\left[\log 2-F\left(\frac{\sqrt{kl}}{\delta}\right)\right]
 \ge \delta[\log 2-F(k/\delta)]+\delta[\log 2-F(l/\delta)]
 \ge h_\delta(k)+h_\delta(l).
\]
This proves the first assertion and its equality condition.

For the leaf case fix $b\ge3$, put $a=\sqrt b/\delta$, and differentiate in $\delta\ge b$. By \eqref{eq:F-derivative-bound},
\begin{equation}\label{eq:leaf-degree-derivative}
 \begin{aligned}
 \partial_\delta\cE_\delta(1,b)
 &=2\log 2-1-\frac1b-2[F(a)-aF'(a)]+\frac2\delta\\
 &\ge2\log 2-1-\frac1b+\frac{2-\sqrt b}{\delta}.
 \end{aligned}
\end{equation}
For $3\le b\le4$ the last line is at least $2\log 2-4/3>0$.
Now suppose $b\ge4$. Since $2-\sqrt b\le0$ and $\delta\ge b$,
$
 (2-\sqrt b)/\delta\ge (2-\sqrt b)/b
$.
Consequently, the last line of \eqref{eq:leaf-degree-derivative} is bounded
below by
\[
 2\log 2-1-\frac1{\sqrt b}+\frac1b\geq 2\log2-\frac{5}{4}>0.
\]
Here we use $-\frac1{\sqrt b}+\frac1b\geq -\frac{1}{4}$ and \eqref{eq:coarse-logs}.
\[
 -\frac1{\sqrt b}+\frac1b=-(r-r^2)\ge-\frac14,
\]
Therefore
\[
 \cE_\delta(1,b)\ge\cE_b(1,b)=2b\,\mathscr R(b^{-1/2})>0
\]
by Lemma~\ref{lem:leaf-barrier}.
\end{proof}

\bibliographystyle{amsplain}
\bibliography{references}
\clearpage
\end{document}

%% file: figures/q4HammingWeightQuotient.tex

\begin{center}
\begin{tikzpicture}[
  x=0.9cm,
  y=0.8cm,
  graph edge/.style={draw=black,line width=0.55pt},
  graph vertex/.style={circle,draw=black,fill=white,line width=0.55pt,
    minimum size=4.8pt,inner sep=0pt},
  quotient arrow/.style={->,draw=black,line width=0.55pt}
]
\begin{scope}[x=0.75cm,y=0.62cm]
  \coordinate (hc-p) at (0,0);
  \coordinate (hc-q) at (6,0);
  \foreach \i/\y in {1/1.5,2/0.5,3/-0.5,4/-1.5}{
    \coordinate (hc-a\i) at (1.35,\y);
  }
  \foreach \i/\y in {4/1.5,3/0.5,2/-0.5,1/-1.5}{
    \coordinate (hc-b\i) at (4.65,\y);
  }
  \foreach \e/\y in {12/2,13/1.2,14/0.4,23/-0.4,24/-1.2,34/-2}{
    \coordinate (hc-e\e) at (3,\y);
  }
  \foreach \i in {1,2,3,4}{
    \draw[graph edge] (hc-p)--(hc-a\i);
    \draw[graph edge] (hc-b\i)--(hc-q);
  }
  \foreach \e/\i/\j in {12/1/2,13/1/3,14/1/4,23/2/3,24/2/4,34/3/4}{
    \draw[graph edge] (hc-a\i)--(hc-e\e);
    \draw[graph edge] (hc-a\j)--(hc-e\e);
  }
  \foreach \e/\i/\j in {12/3/4,13/2/4,14/2/3,23/1/4,24/1/3,34/1/2}{
    \draw[graph edge] (hc-e\e)--(hc-b\i);
    \draw[graph edge] (hc-e\e)--(hc-b\j);
  }
  \foreach \v in {hc-p,hc-a1,hc-a2,hc-a3,hc-a4,
    hc-e12,hc-e13,hc-e14,hc-e23,hc-e24,hc-e34,
    hc-b1,hc-b2,hc-b3,hc-b4,hc-q}{
    \node[graph vertex] at (\v) {};
  }

  \node at (0,2.8) {$S_0$};
  \node at (1.35,2.8) {$S_1$};
  \node at (3,2.8) {$S_2$};
  \node at (4.65,2.8) {$S_3$};
  \node at (6,2.8) {$S_4$};
  \node at (3,-2.7) {$4$-dimensional hypercube $H_4$};
\end{scope}

\draw[quotient arrow] (5.7,0)--(7.1,0)
  node[midway,above=4pt] {$S_i\mapsto i$};

\begin{scope}[xshift=7.6cm]
  \coordinate (p-0) at (0,0);
  \coordinate (p-1) at (1.5,0);
  \coordinate (p-2) at (3,0);
  \coordinate (p-3) at (4.5,0);
  \coordinate (p-4) at (6,0);
  \draw[graph edge] (p-0)--(p-1)
    node[midway,above=3pt] {$b=4$};
  \draw[graph edge] (p-1)--(p-2)
    node[midway,above=3pt] {$b=12$};
  \draw[graph edge] (p-2)--(p-3)
    node[midway,above=3pt] {$b=12$};
  \draw[graph edge] (p-3)--(p-4)
    node[midway,above=3pt] {$b=4$};
  \foreach \v in {p-0,p-1,p-2,p-3,p-4}{
    \node[graph vertex] at (\v) {};
  }
  \node at (0,-0.55) {$m=1$};
  \node at (1.5,-0.55) {$m=4$};
  \node at (3,-0.55) {$m=6$};
  \node at (4.5,-0.55) {$m=4$};
  \node at (6,-0.55) {$m=1$};
  \node at (3,-2.1) {weighted $P_5$};
\end{scope}
\end{tikzpicture}
\end{center}